\documentclass[10pt]{article} 
\usepackage[margin=1in]{geometry}
 
\usepackage{listings}
\usepackage{geometry}

\usepackage{amsmath,amsthm,amsfonts,amssymb,amscd}
\usepackage[colorlinks=true, pdfstartview=FitV, linkcolor=blue,citecolor=blue, urlcolor=blue]{hyperref}
\usepackage[usenames,dvipsnames]{color}
\usepackage{times}

\usepackage{mathtools}
\usepackage{mathabx}
\usepackage{float}
\usepackage{makeidx}
\usepackage{stmaryrd}
\usepackage{mathrsfs}
\usepackage{emptypage}
\usepackage{xfrac}
\usepackage{bbm}
\usepackage{comment} 
\usepackage{xcolor} 

\usepackage{physics}

\newtheorem{theorem}{Theorem}[section]
\newtheorem{lemma}[theorem]{Lemma}
\newtheorem{corollary}[theorem]{Corollary}

\newtheorem{remark}[theorem]{Remark}
\newtheorem{conjecture}[theorem]{Conjecture}

\title{A note on the periodic Hilbert Transform on a strip}
\date{\today}
\author{Javier G\'omez-Serrano, Sieon Kim}

\begin{document}

\maketitle

\begin{abstract}
In this note we prove a conjecture by Constantin--Strauss--V\u{a}rv\u{a}ruc\u{a} 
related to the finite depth water wave problem, tightening their results. The proof uses identities related to Jacobi Theta functions. We also discuss potential implications of the improvement.
\end{abstract}

\section{Introduction}

In this paper we will be concerned with the periodic Hilbert transform on a strip. Let $d > 0$ and let
\[\mathcal{R}_{d}=\{(x,y) \in \mathbb{R}^2:\ -d < y < 0\}\]
be the strip of depth $d$. Then, if $w \in C^{0,\alpha}(\mathbb{R})$, is $2\pi$-periodic, has zero mean and the Fourier series expansion
\[w(x)=\sum_{n=1}^\infty a_n \cos (nx)\,+\,\sum_{n=1}^\infty b_n \sin (nx),\qquad x \in \mathbb{R},\]
the Hilbert transform operator $\mathcal{C}_d$ acts as
\begin{equation}\label{hfs}
\big(\mathcal{C}_d(w)\big)(x)=\sum_{n=1}^\infty a_n \coth
 (nd)\sin(nx)-\sum_{n=1}^\infty b_n\coth (nd)\cos (nx), \qquad x \in \mathbb{R}\,.
\end{equation}

The Hilbert transform is of central importance in the study of the water wave problem, since it arises in the context of the Dirichlet-Neumann operator and the linearized equation around a flat wave. See 
\cite{King:hilbert-transforms-i,King:hilbert-transforms-ii} for further references on the Hilbert Transform and \cite{Haziot-Hur-Strauss-Toland-Wahlen-Walsh-Wheeler:water-waves-survey,Lannes:water-waves-book} for comprehensive surveys on the water wave problem.

In \cite{Constantin-Strauss-Varvaruca:large-amplitude-steady-downstream-water-waves}, the authors conjectured the following:

\begin{conjecture}[\cite{Constantin-Strauss-Varvaruca:large-amplitude-steady-downstream-water-waves} Lemma 1, Remark, p.252, abridged, see also \cite{Haziot-Strauss:amplitude-bounds-steady-rotational-water-waves} Lemma 2)]

Let us recall from Appendix A in \cite{Constantin-Strauss-Varvaruca:global-bifurcation-gravity-water-waves} that, for
any smooth $2\pi$-periodic function $F: {\mathbb R} \to {\mathbb R}$ with mean zero over each period,
we have
\begin{equation}\label{fc0}
(\mathcal{C}_{kh}(F))(x)=\frac{1}{2\pi} \text{PV} \int_{-\pi}^\pi \beta_{kh}(x-s)F(s)\,ds\,,\qquad x \in \mathbb{R}\,,
\end{equation}
where (with $d=kh$) the kernel $\beta_{d}: \mathbb{R} \setminus 2\pi {\mathbb Z}\to {\mathbb R}$,
is given by
\begin{align}  \label{betadef}
\beta_{d}(s) &= -\frac{s}{d}  +  \frac{\pi}{d}\coth\Big(\frac{\pi s}{2d}\Big)  +  \frac{\pi}{d} \sum_{n=1}^{\infty}
\frac{2\sinh(\frac{\pi s}{d})}  {\cosh(\frac{\pi s}{d})-\cosh(\frac{2\pi^2n }{d})}\\
&=  -\frac{s}{d} + \frac\pi{d}  \sum_{n=-\infty}^\infty  \left\{ \coth\left(\frac\pi{2d}(s-2\pi n)\right) + \text{sgn}(n) \right\}.\nonumber
\end{align}

Moreover, let us write $x=\pi^2/2d$ and consider $\beta_d(\pi/2)$
 as a function of $x\in(0,\infty)$.

Since letting $d\to \infty$ is equivalent to letting $x\to 0$ in the formula
$$
\pi\beta_d(\pi/2)  =  2x\coth(x/2) -x  -  4x\sinh(x) \sum_{n=1}^\infty \frac1{\cosh(4n x)-\cosh(x)},  $$
we get
\[\pi\lim_{d\to \infty}\beta_d(\pi/2)=4-8\sum_{k=1}^{\infty}\frac{1}{16n^2-1}=\pi\,.\]
It is an interesting conjecture whether actually
\[\beta_d(\pi/2)\geq 1\quad\text{for all }d\in (0,\infty)\,.\]
Numerical computation in Octave/Matlab suggests that the issue is quite subtle.
    \label{main_conjecture}
\end{conjecture}

Our main result in this note is the following:

\begin{theorem}
\label{main_thm}
    Conjecture \ref{main_conjecture} is true.
\end{theorem}

\begin{remark}
The subtlety in this conjecture lies in the fact that the function grows extremely slowly for small values of $x$, as seen in Figure \ref{fig:beta}. Indeed, it is difficult to obtain a lower bound for the function via conventional means (such as for example Taylor expansions) since $\beta_d(\pi/2)$ remains so close to $1$, corroborated by how $x=0.1$ evaluates to $1$ with $42$ digits of accuracy according to Mathematica simulations. Hence, it is a non-trivial task to estimate the function's growth around $0$. Instead we establish the lower bound by proving the equivalence between the function $\beta$ and a Jacobi Theta function, and then establishing monotonicity.
\begin{figure}[h]
    \centering    \includegraphics[scale=0.8]{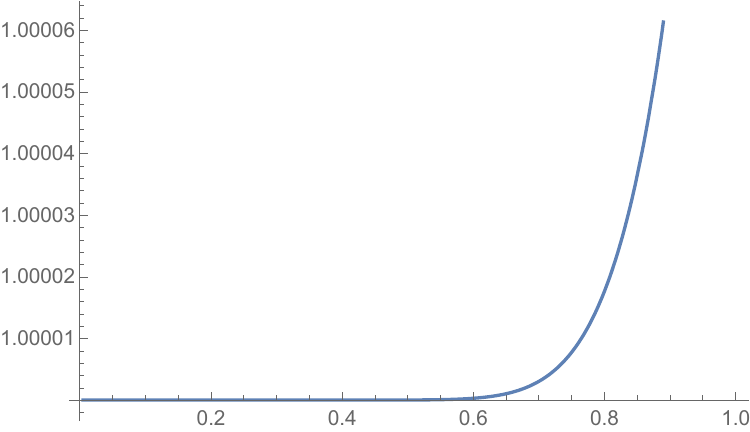}
    \caption{The function $\beta_d\left(\frac{\pi}{2}\right)$ as a function of $x$.}
    \label{fig:beta}
\end{figure}

\end{remark}


We defer the proof to Section \ref{section:mainthm}. We now mention several small improvements as corollaries of the theorem. They are related to better bounds on some of the constants implied by the better bound on the function $\beta_d\left(\frac{\pi}{2}\right)$. For simplicity we refer to the corresponding papers for the definitions:

\begin{corollary}[Strengthening of Theorem 4 of \cite{Constantin-Strauss-Varvaruca:large-amplitude-steady-downstream-water-waves}]

\label{bound} Let $\Upsilon\geq 0$. Then,
along the whole global bifurcation curve ${\mathcal K}_-$, we have the estimate
\begin{subequations}\label{Fii}
\begin{equation}
v(0)-v(\pi) \le \sqrt{\frac{36g^2}{\Upsilon^4}+\frac{24\pi g}{\Upsilon^2k\beta_{hk}(\tfrac{\pi}{2})}} -\frac{6g}{\Upsilon^2} \le \sqrt{\frac{36g^2}{\Upsilon^4}+\frac{24\pi g}{\Upsilon^2k}} -\frac{6g}{\Upsilon^2}  
\quad \text{if}\quad \Upsilon>0,
\end{equation}
and
\begin{equation} v(0)-v(\pi) < \frac{2\pi}{k\beta_{kh}(\tfrac{\pi}{2})} \le
\frac{2\pi}{k}\quad\text{if}\quad \Upsilon=0. \end{equation}
\end{subequations}

\end{corollary}

\begin{corollary}[Strengthening of Theorem 1.3 of \cite{Haziot-Strauss:amplitude-bounds-steady-rotational-water-waves}]

	Consider a smooth water wave that belongs to the bifurcation curve $\mathcal C$ 
	in the adverse case $\gamma>0$ and assume that \textit{either} the slope $|\eta'|$ 
	or the convexity $|\eta''|$ of the wave is bounded.  
	Then the wave amplitude $\mathscr{A}$ (the elevation difference between the crest and trough) 
	is uniformly bounded by a certain constant provided $\gamma$ is sufficiently small.  
	The upper bound depends only on a certain explicit function of 
	the constants $g, Q, m$ and the conformal depth $d$.  
	
	As an example, the upper bound can be chosen to be ${12\pi}$, if we assume that each of the  quantities $\gamma, Q\gamma, |m|\gamma^2$, as well as either $N\gamma^2$ or  $M\gamma^4$, 
are less than certain explicit functions of $g$ and $d$.
\end{corollary}

\section{Proof of Theorem 
\ref{main_thm}}
\label{section:mainthm}


\begin{proof}
We will show that $\beta_d\left(\frac{\pi}{2}\right) \geq 1$ for all $d \in (0,\infty)$ where $x=\pi^2/d$ and
\begin{align}
\label{betapidefinition}
    \pi\beta_d(\frac{\pi}{2})=2x\coth \left(\frac{x}{2} \right)-x-4x\sinh(x)\sum_{n=1}^\infty \frac{1}{\cosh(4nx)-\cosh(x)}.
\end{align}
We will also deduce that $\beta_d\left(\frac{\pi}{2}\right)$ is strictly increasing in $x$ for $x>0$.

We use the following intermediate lemma to begin our proof.

\begin{lemma}
$$\begin{aligned}
\sum_{n=1}^\infty \frac{2\sinh(x)}{\cosh(4nx)-\cosh(x)} 
&=\sum_{n=1}^\infty \frac{\sinh(x)}{\sinh(\frac{4n+1}{2}x)\sinh(\frac{4n-1}{2}x)}.\\
\end{aligned}$$
\end{lemma}
\begin{proof}
    Use the difference to product rule for hyperbolic cosine and simplify:
    $$\cosh(a)-\cosh(b)=2\sinh\left(\frac{a+b}{2}\right)\sinh\left(\frac{a-b}{2}\right).$$
\end{proof}

Let $q=e^{-x}$. Denoting the sum by $2S$ and expanding in terms of $q$, we obtain
$$2S=\sum_{n=1}^\infty \frac{\sinh(x)}{\sinh(\frac{4n+1}{2}x)\sinh(\frac{4n-1}{2}x)}=2\sum_{n=1}^\infty \frac{(q^{-1}-q)}{(q^{-(4n+1)/2}-q^{(4n+1)/2})(q^{-(4n-1)/2}-q^{(4n-1)/2})}.$$
\color{black}Simplifying by multiplying $q^{(4n+1)/2}q^{(4n-1)/2}$ in the numerator and denominator, we obtain
$$\sum_{n=1}^\infty \frac{(q^{4n-1}-q^{4n+1})}{(1-q^{4n+1})(1-q^{4n-1})}=\sum_{n=1}^\infty \frac{((q^{4n-1}-1)+(1-q^{4n+1}))}{(1-q^{4n+1})(1-q^{4n-1})},$$
which leads to
$$S=\sum_{n=1}^\infty \frac{\sinh(x)}{\cosh(4nx)-\cosh(x)}=\sum_{n=1}^\infty \left(\frac{1}{1-q^{4n-1}}-\frac{1}{1-q^{4n+1}}\right)=\sum_{n=1}^\infty \left(\frac{q^{4n-1}}{1-q^{4n-1}}-\frac{q^{4n+1}}{1-q^{4n+1}}\right).$$

By \cite[Theorem 259, Chapter 12]{Knopp:infinite-series-book}, a series of the form $\sum_k a_kz^k/(1-z^k)$ (known as a \textit{Lambert series}) converges if $\sum_k a_kz^k$ converges and $|z| \neq 1$. This implies that since $q \in (0,1)$, $\sum_kq^k/(1-q^k)$ converges since $\sum_k q^k$ converges, and for the same reason so do $\sum_n q^{4n-1}/(1-q^{4n-1})$ and $\sum_n q^{4n+1}/(1-q^{4n+1})$ for $x\in (0,\infty)$, implying that the whole series converges for $x \in (0,\infty)$. As a result, we may split and rearrange the terms as follows: 
$$\begin{aligned}
S=\sum_{n=1}^\infty \left(\frac{q^{4n-1}}{1-q^{4n-1}}-\frac{q^{4n+1}}{1-q^{4n+1}}\right)&=\left(\frac{q^3}{1-q^3}-\frac{q^5}{1-q^5}\right)+\left(\frac{q^7}{1-q^7}-\frac{q^9}{1-q^9}\right)+\cdots\\
&=\frac{q}{1-q}-\left(\left(\frac{q}{1-q}-\frac{q^3}{1-q^3}\right)+\left(\frac{q^5}{1-q^5}-\frac{q^7}{1-q^7}\right)+\cdots\right)\\
&=\frac{q}{1-q}-\sum_{n=1}^\infty \left(\frac{q^{4n-3}}{1-q^{4n-3}}-\frac{q^{4n-1}}{1-q^{4n-1}}\right).
\end{aligned}$$
Plugging this back into \eqref{betapidefinition} and expanding each term in $q$, we obtain
$$\begin{aligned}
\pi\beta_d\left(\frac{\pi}{2}\right)&=x\left[2\coth\left(\frac{x}{2}\right)-1-4S\right]\\
&=x\left[2\left(\frac{1+q}{1-q}\right)-1-4\left(\frac{q}{1-q}-\sum_{n=1}^\infty \left(\frac{q^{4n-3}}{1-q^{4n-3}}-\frac{q^{4n-1}}{1-q^{4n-1}}\right)\right)\right],
\end{aligned}$$
and simplifying this expression yields
\begin{align}
\label{eq:betarepresentation}
\pi\beta_d\left(\frac{\pi}{2}\right)=x\left[ 1+4\sum_{n=1}^\infty \left(\frac{q^{4n-3}}{1-q^{4n-3}}-\frac{q^{4n-1}}{1-q^{4n-1}}\right)\right].
\end{align}
We now introduce the elliptic theta function $\vartheta_3(z|\tau)$, defined in \cite{Mumford:tata-theta-i} as
$$\vartheta_3(z|\tau)=\sum_{n=-\infty}^\infty e^{\pi i \tau n^2+2\pi i n z},$$
where $z \in \mathbb{C}$ and $\tau \in \mathbb{H},$ the upper half complex plane. Notice that setting $z=0$ and writing $q=e^{\pi i \tau}$ gives the representation
\begin{align}
\label{eq:theta3_series}
\vartheta_3(0,q)=\vartheta_3\left(0|\tau \right)=\sum_{n=-\infty}^\infty q^{n^2},
\end{align}
with $\tau=ix/\pi$ having positive imaginary part for $x>0$. By \cite[Theorem 312, Chapter XVII]{Hardy-Wright:introduction-theory-numbers} (see \cite{Jacobi:fundamenta} for the original proof), we have that 
$$\left(\sum_{n=-\infty}^\infty q^{n^2}\right)^2=1+4\sum_{n=1}^\infty \left(\frac{q^{4n-3}}{1-q^{4n-3}}-\frac{q^{4n-1}}{1-q^{4n-1}}\right),$$
and since the right hand side equals $x^{-1}\pi \beta_d(\pi/2)$ by \eqref{eq:betarepresentation}, we have
$$\pi\beta_d\left(\frac{\pi}{2}\right)=x\vartheta_3^2(0,q).$$

Additionally, using \cite[Equation (9.2)]{Bellman:introduction-theta-functions}
$$\sum_{n=-\infty}^\infty e^{-xn^2}=\sqrt{\frac{\pi}{x}}\sum_{n=-\infty}^\infty e^{-\pi^2 n^2/x},$$
and the definition of $\vartheta_3$, this implies
$$x\vartheta_3^2(0,e^{-x})=\pi \vartheta_3^2\left(0,e^{-\frac{\pi^2}{x}}\right).$$

Hence, we obtain
$$\beta_d\left(\frac{\pi}{2}\right)=\vartheta_3^2\left(0,e^{-\frac{\pi^2}{x}}\right).$$
 

Monotonicity then follows from \eqref{eq:theta3_series}.
Recalling that $d=\pi^2/x$ and
$$\lim_{d \to \infty}\beta_d\left(\frac{\pi}{2}\right)=1,$$
we deduce that $\beta_d(\pi/2) \geq 1$ for $d \in (0,\infty)$.


\end{proof}

\section*{Acknowledgements}
JGS was partially supported by NSF through Grants DMS-2245017 and DMS-2247537; by the AGAUR
project 2021-SGR-0087 (Catalunya) and by the MICINN
(Spain) research grant number PID2021-125021NA-I00. We are very grateful to Susanna Haziot and Walter Strauss for many stimulating discussions and constructive comments.

\bibliographystyle{abbrv}
\bibliography{references}

\begin{tabular}{l}
\textbf{Javier G\'omez-Serrano}\\
{Department of Mathematics} \\
{Brown University} \\
{314 Kassar House, 151 Thayer St.} \\
{Providence, RI 02912, USA} \\
{Email: javier\_gomez\_serrano@brown.edu} \\ \\
\textbf{Sieon Kim}\\
{Seoul International School} \\
{15, Seongnam-daero,}\\
{1518 beon-gil, Sujeong-gu,}\\
{Seongnam-si, Gyeonggi-do,}\\
{South Korea 13113}\\
{Email: sieon.kim25@stu.siskorea.org}\\ \\
\end{tabular}
\end{document}